\documentclass[12pt]{article}
\usepackage{latexsym,color,amsmath,amsthm,amssymb,amscd,amsfonts,bbm,verbatim}
\usepackage[all]{xy}
\newtheorem{lemma}{lemma}[section]

\newtheorem{remark}[lemma]{Remark}
\newtheorem{theorem}[lemma]{Theorem}
\newtheorem{definition}[lemma]{Definition}

\newtheorem*{remark*}{Remark}

\usepackage{mathabx,epsfig}
\def\acts{\mathrel{\reflectbox{$\righttoleftarrow$}}}

\begin{document}
\title {\bf{A Family of Quasimorphism Constructions}}
\author{Gabi Ben Simon
\\
ETH-Z\"urich\\
 gabi.ben.simon@math.ethz.ch} \maketitle

\begin{abstract}

 In this work  we present a principle  which says that quasimorphisms can be obtained via "local data" of the group action on certain appropriate spaces. In a rough manner the principle says that instead of starting with a given group and try  to build or study its space of quasimorphisms,  we should start with a space with a certain structure, in such a way that groups acting on this space and respect this structure will automatically carry quasimorphisms,  where these are suppose to be better understood.  In this paper we suggest such a family of spaces and give demonstrating examples  for countable groups, groups that relates to action on the circle as well as outline construction for diffeomorphism groups. A distinctive advantage of this principle is that it allows the construction of the quasimorphism in a quite direct way.  Further,  we prove a lemma which besides serving as a platform for the construction of quasimorphisms on countable groups, bare interest by itself. Since it provides us with an embedding of any given countable group as a group of quasi-isometries of a universal space,  where this space of embeddings is in bijection with the projective space of the homogeneous quasimorphism space of the group.

\end{abstract}

 \section{Introduction}
 
 Given a group, $G$,  a \textit{quasimorphism} on the group, $\mu$,  is a function to $\mathbb{R}$ which satisfies  
 
 $$|\mu(xy)-\mu(x)-\mu(y)| \leq B$$ for all $x, y \in G$
 and a universal $B$. The homogenization of $\mu$, $\mu^{h}(g):= \lim \limits_{n \rightarrow \infty} \frac{\mu(g^{n})}{n}$, 
 where the limit exists, is a qusimorphism with bounded difference from $\mu$. Further $\mu^h$,  is a homogeneous function which means $\mu^{h}(g^n)=n \mu^{h}(g)$ for every integer $n$.

 From the point of view of the author's interest  there are two main sources of study of this notion. The one that comes from the attempt to construct quasimorphisms on diffeomorphisms group, of special interest symplectomorphisms and Hamiltonian groups  (see for example \cite{BS}, \cite{EnP}, \cite{Py}, \cite{Shel}).  Where the other source comes from the study of Lie groups and,  with and without relation,  countable groups. So the groups studied are, for example, universal covers of hermitian Lie groups and word hyperbolic groups   (see for example \cite{BS-H3}, \cite{BS-H1}, \cite{BM}, \cite{C},  \cite{CF}, \cite{BIW}, \cite{EP} and \cite{calegari} as a general reference).  In this work  we want to report a feature, which appears sometimes  indirectly,  which says that quasimorphisms can be obtained via "local data" of the group action on certain appropriate spaces. This feature appears in both families mentioned above.  In a rough manner the principle says that instead of starting with a given group and try  to build or study its space of quasimorphisms,  we should start with a space with a certain structure, in such a way that groups acting on this space and respect this structure will automatically carry quasimorphisms, where these should be quite understood. In this paper we suggest such a family of spaces and give demonstrating examples  for countable groups, groups that relates to action on the circle as well as outline construction for diffeomorphism groups, see section 3. A distinctive advantage of this principle is that it allows the construction of the quasimorphism in a quite direct way.  Further, see subsection \eqref{ladder}, we prove a lemma which besides serving as a platform for the construction of quasimorphisms on countable groups, bare interest by itself. Since it provides us with an embedding of the countable group as a group of quasi-isometries of a universal space,  where this space of embeddings is in bijection with the projective space of the homogeneous quasimorphism space of the group. We see this paper as a first step, in developing the picture that emerge from it.
 
 We should remark that the idea to look at quasimorphisms  from the point of view of group action on spaces with certain, appropriate, structure has started in \cite{BS-H4}. Nevertheless the focus there was completely different. Indeed the focus was on a systematic study of the relation of  quasimorphisms and the notion of relative growth and order structures on groups,  see \cite{BS-H4}.

\textbf{Acknowledgements:} 

Many thanks to Danny Calegari for his willingness to host me at Cambridge University given  a very short notice. The discussion with him was important for this work. Many thanks to Tobias Hartnick for reading a draft of the paper and for his remarks. Many thanks to Andreas Leiser for the help with the drawings. Many thanks to Leonid Polterovich for his important remarks about the preliminary version. The remarks of Dietmar Salamon about the diffeomorphism group helped me to focus my intension and ideas I thank him very much for that. Finally I am grateful to the Departement Mathematik of ETH Z\"urich for the support  during this academic year, and in particular to  Paul Biran and Dietmar Salamon.

\newpage
\section{Main Principle}

The starting point is the following simple fact that can be extracted from $ \cite {BS-H4}$.

 \begin{lemma}\label{basiclemma}  Assume that   $X$   is a space such that there exists a function $h: X \rightarrow  \mathbb{R} $ and a group action 
 $G \acts X$ such that for any $ x,y \in X, g \in G $ we have: 

\begin{equation} \label{root condition}  |(h(g \cdot x)-h(g \cdot y))-(h(x)-h(y))| \leq B  \end{equation}

for some universal bound  B .

Then the function   $ \mu (g)= h(g \cdot a)- h(a) $  is a  quasimorphism,  where  $\mu$ does not depend on the choice of  $a \in X $ up to a bounded error.
Further, if the action is effective and is not bounded in the sense that  
  $  \lim \limits_{n \to \pm \infty} h(g^n \cdot a) = \pm \infty $ for some $g$,  then $\mu ^h$, the homogenization of $ \mu$ is a non zero homogeneous quasimorphism. 
  
\end{lemma}

 \begin{remark}:\label{br} Actually, as it shown in \cite{BS-H4}  every homogeneous quasimorphism can be obtained in this way by simply choosing $X=G$ and $h$ to be the given  quasimorphism on $G$.
 \end{remark}

 The lemma suggests a usage of the "inverse ideology", meaning: start with a space $h: X \rightarrow \mathbb{R} $ and try to find a group action $G \acts X $ which satisfies condition  \eqref{root condition}  above. The paper suggests one possible answer to this problem.
 
 We now give a setup which will lead to examples.
 
 Let $X$ be a space with no special structure. And assume that a group $ A $ acts on $X$ such that \footnote{Choosing transparency over conciseness, we choose to skip what might be a more concise formulation, using standard terminology, of the axioms.}

\bigskip

$ \begin{cases}\label{mainsetup}
  X=   \coprod \limits_{ \alpha \in A } F_{ \alpha } & (X  \text{ is a union of "fundamental domains"})\\
  \alpha : F_{ \mathbbm{1}} \rightarrow F_{ \alpha }, & \forall \alpha \in A \\
   h: X \rightarrow \mathbb{R} , & \text{ s.t. }  Im(h(F_{\mathbbm {1}})) \subseteq [0,1), \\
    h( \alpha (x))=h(x)+ \rho (\alpha)+b(x, \alpha), & \forall x \in F_{ \mathbbm{1}}, \forall \alpha \in A
  \end{cases}    $

\noindent where the restriction of $\alpha$ above is a bijection, $h$ a function on $X$,  $b: F_{ \mathbbm{1}} \times A \rightarrow \mathbb{R}$ is a bounded function and,  $\rho$ an unbounded homomorphism to $\mathbb{Z}$.

 
 
 
\begin{definition}\label{triple} We define $X$,  $h$ and $A$ compatible as above as a triple, and denote it by $( X, h , A )$. We do not include $b$ and $\rho$ in the notation since in all that follows they will not be used directly.
\end{definition}
 
\begin{theorem}\label{maintheorem} Assume that $G$ acts on the space $( X, h, A )$ such that the action of $G$ commutes with the action of $A$ and, $h(g(F_{ \mathbbm{1}})) \subseteq [r, r+C_{0}]$ for all $g \in G$, and  
  for universal constant $C_{0}$,  and some $r$ which depends on $g$. For example we can choose $ r:= \inf(Im(h(g(F_{}))) $.
 Then $\mu$ of Lemma \ref{basiclemma} defines a  quasimorphism.   
 \end{theorem}

\begin{remark}\label{ac}
 
 1) The assumption that $A$ and $G$ need to commute can be relaxed. Instead we can demand that
 the groups actions will \textit{ almost commute}, which means that for all $\alpha \in A$ and $ g \in G$ we have that $| h(\alpha \circ g\cdot x_{0})-h(g \circ \alpha\cdot x_{0})|$ is universally bounded independently of $x_{0}$. The proof goes exactly the same.

 2) Note that due to the facts that $X$ is "tiled" by images of $F_{ \mathbbm{1}}$ under the action of $A$, and $A$ commutes with $G$, it is enough to construct, or define, the image of  $F_{ \mathbbm{1}}$ under
 the restriction of the action of $G$ to $F_{ \mathbbm{1}}$.
 
 \end{remark}

 \begin{proof} The proof is made, essentially,  of three simple facts: The first thing to note is that for every $\gamma \in A$ we have that $h(F_{\gamma}) \subseteq [r, r+M_{0}]$ for some $r$ and the universal constant $M_{0}$ which bounds $b$. Indeed we have $h(F_{\gamma})=h(\gamma (F_{\mathbbm{1}}))$ and it follows from the last axiom that \begin{equation}\label{p1} h( \gamma \cdot x)=h(x)+ \rho (\gamma)+b(x, \gamma) \end{equation} so the claim follows since $b$ is bounded and of course the bounds do not depend on $\gamma$. The second thing to note is that it follows from the axioms that for all $g \in G,\text{ } \gamma \in A, \text{ } x \in X$ we have that \begin{equation}\label{p2} |h(g \circ \gamma \cdot x)-(h(g\cdot x)+ \rho (\gamma)+b)|\end{equation}  is universally bounded (actually by $2M_{0}$ ) where we use \eqref{p1} to see it. 
 The third thing to notice is that it follows from the assumption that $(h(g(F_{ \mathbbm{1}})) \subseteq [r, r+C_{0}]$
and the fact that the actions of $A$ and $G$ commute,  that for all $\alpha \in A$ we have \begin{equation}\label{p3} (h(g(F_{ \alpha})) \subseteq [r, r+C_{0}]\end{equation} for some $r$.

 So now let $x,y \in X \text {and }g\in G$. Further assume that $x \in F_{\alpha}$ and $y \in F_{\beta}$ so of course we have $\beta\circ \alpha^{-1} \cdot x \in F_{\beta}$. We now estimate

  $$ |(h(g \cdot x)-h(g \cdot y))-(h(x)-h(y))|=$$
 $$ |(h(g \cdot x)-h(g\circ \beta \alpha^{-1}\cdot x)+
 h(g\circ \beta \alpha^{-1}\cdot x)
  -h(g \cdot y)+ h(\beta \alpha^{-1}\cdot x)-h(x)+h(y)- h( \beta \alpha^{-1}\cdot x)| $$
 
 \begin{align}\label{b1} \leq & |(h(g \cdot x)-h(g\circ \beta \alpha^{-1}\cdot x) +h(g\circ \beta \alpha^{-1}\cdot x)\\\label{b2} & \nonumber-h(g \cdot y)+  h(\beta \alpha^{-1}\cdot x)-h(x)| \\  &+ |h(y)- h( \beta \alpha^{-1}\cdot x)| \leq 
  \end{align}

\begin{align}  \label{t1} & |(h(g \cdot x)-(h(g \cdot x)+\rho( \beta \alpha^{-1})+b) +h(g\circ \beta \alpha^{-1}\cdot x)\\  \label{t2} & \nonumber-h(g \cdot y)+  (h(x)+\rho( \beta \alpha^{-1})+b)-h(x)| \\  &+ |h(y)- h( \beta \alpha^{-1}\cdot x)|+2M_{0} \leq\\ \label{t3} & | h(g\circ \beta \alpha^{-1}\cdot x)-h(g \cdot y) | \\  \label{t4} &+ |h(y)- h( \beta \alpha^{-1}\cdot x)|+4M_{0} \leq 4M_{0}+1+C_{0}
\end{align}

The transition from \eqref{b1} and \eqref{b2}   to \eqref{t1} and \eqref{t2} comes from the second and the first properties above (which come from  \eqref{p2} and  \eqref{p1} respectively). Further, the transition to \eqref{t3} and \eqref{t4} comes from the fact that $b$ is globally bounded and from the third property above \eqref{p3}. This concludes the proof.

 
 \end{proof}
 
 \section{Examples demonstrating the main principle}
 
 \textit{From now on we set:  $A= \mathbb{Z} $ and $ \rho = id$, unless otherwise stated. We use the same notation and conventions as above.}

We now move to the next family. Using theorem \eqref{maintheorem} with an eye to countable groups.

\subsection{Family 1: Countable Groups}

\subsubsection{Example}

It follows almost directly from the results of \cite{BS-H4} that the Rademacher quasimorphism 
on $PSL_{2}( \mathbb{Z})$ fits into the scheme  of Theorem \eqref{maintheorem}. Here again we have $A = \mathbb{Z} $ and $X$ is a countable 
subsets of $ \mathbb{R}^2 $ and $h: X \rightarrow \mathbb{R}$. Further, we have an action of $ \mathbb{Z}$ on $X$ which almost  commutes (see Remark\eqref{ac}) with the action of $PSL_{2}( \mathbb{Z})$ on $X \subseteq \mathbb{R}^2$. The resulting quasimorphism, as we said, is a Rademacher  quasimorphism. The details are as follows, we repeat the data of \cite{BS-H4} in which the full details of the construction appear. Recall that $PSL_{2}( \mathbb{Z}) \cong \mathbb{Z}_2 \ast  \mathbb{Z}_3$. Under this isomorphism we denote by $S$ and $R$ the generators of $\mathbb{Z}_2$ and $ \mathbb{Z}_3$ respectively. In figure \ref{fig2} below we demonstrate how the group acts on a subset,  $X$ (where of course only part of it appears in the figure), of the plane where the action is obvious from the figure. Here $h$ increases in the horizontal direction- again see the figure. Lastly,  between the two dotted lines we have one fundamental domain of the action.

\begin{figure}[h!t]
\begin{center}
\includegraphics[width=2.5in]{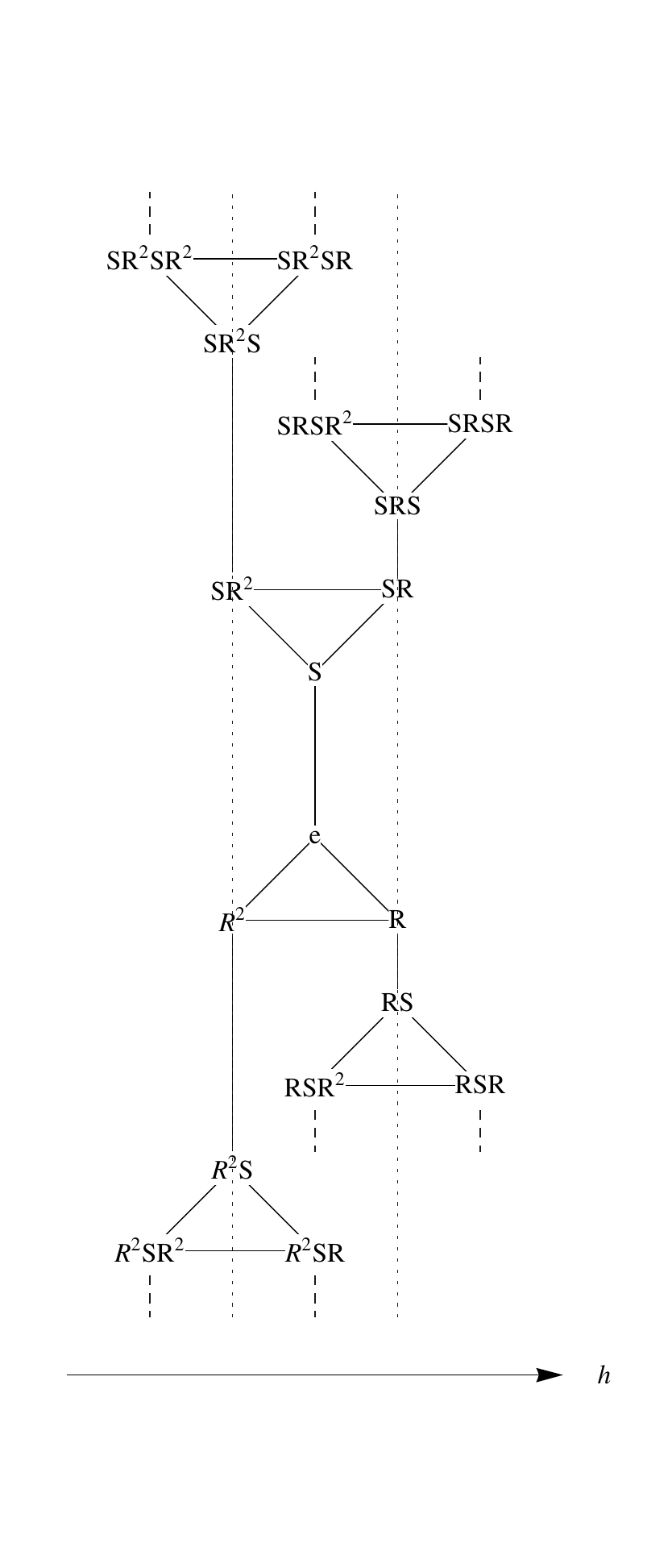}\caption{The $PSL_2(\mathbb{Z}$) example}\label{fig2}
\end{center}
\end{figure}
\newpage

\subsubsection{Universal embedding for Countable Groups with nonzero quasimorphism}\label{ladder}

 Further, the following set up, culminating in the lemma below,  also serve as a platform for building examples,  by applying Theorem \eqref{maintheorem}.

Consider a discrete countable set on the open interval $(0,1)$ denoted by $H$. And consider the "ladder" set $ \mathcal{L}:= H\times \mathbb{Z} \subseteq \mathbb{R}^{2}.$ Define a metric on $\mathcal{L}$ as $d=d_{1}+d_{2}$ where the $d_{i}$ are the usual induced metrics on $H$ and $\mathbb{Z}$ respectively. So we have a metric space $(\mathcal{L}, d)$. Lets denote by $h: \mathcal{L} \rightarrow \mathbb{Z}$ the projection to the $\mathbb{Z}$ component. Finally, we denote by $QI^{h}(\mathcal{L},d)$ the space of all quasi-isometries of  $(\mathcal{L},d)$ which respect condition \eqref{root condition} where $g$, in \eqref{root condition} stands for quasi-isometry and $B$ depends on $g$.

Now let $G$ be any countable group for which the space of nonzero homogenous quasimorphism is not empty. Then we have: 

\begin{lemma}\label{qil}For a given homogeneous quasimorphism on $G$, say $\mu$, there is an injection  of $G$  induced by $\mu$, into $QI^{h}(\mathcal{L},d)$. Further,  if we denote the action of the image of $G$ on $\mathcal{L}$, induced by $\mu$,  by $\Psi^{\mu}$  and two representations are considered to be equivalent if the difference,  with respect to $h$, between the orbits of their action, for any point in $\mathcal{L}$,  is universally bounded. Then, there is an injection of the projective space of the homogenous quasi-morphisms space into the equivalence classes space of the representations. In other words if $\mu$ and $m$ are linearly independent then $[\Psi^{\mu}] \neq [\Psi^{m}]$, where the brackets stand for equivalence classes.

\end{lemma}

\textbf{Remark.}  It is worthwhile to emphasize what are the lemma's main points.

  1. The space $\mathcal{L}$ is quasi-isometric to $ \mathbb{Z}$, nevertheless the level sets of $h$ on  $\mathcal{L}$  plays a very important role in the embedding of $G$ above,  so they can not be discarded.

2. The lemma says that \textbf{any} quasimorphism on any countable group comes from an injection of the group into  $QI^{h}(\mathcal{L},d)$. It further tells us that two injections will be essentially the same if their quasimorphisms are.

3. Note that actually any group in  $QI^{h}(\mathcal{L},d)$, countable or not countable, for a \textit{fixed} constants in the quasi-isometry condition will cary a  quasimorphism.  By using \eqref{basiclemma} . So it means that such a group has at least as many embeddings into $QI^{h}(\mathcal{L},d)$, as the points of the projective space of its homogeneous quasimorphism space. 
For example: For  $SL_{n}( \mathbb{Z})$ for $n \geq 3$, being a boundedly  generated group by its elementary subgroups, in which every element can be considered as a commutator,   it is easily follows that  for any injection of this group to  $QI^{h}(\mathcal{L},d)$ we will have only bounded orbits. This is of course not a surprise knowing that the group has no non trivial homogeneous quasimorphisms.

\begin{proof}

Since $\mu$ is a nonzero homogeneous quasimorphism then we can choose a  quasimorphism $\mu_0: G \rightarrow \mathbb{R}$ with the following properties: (i) The homogenization of   $\mu_0 \text { is } \mu $, (ii) $ \mu^{-1}_{0}(n) \neq \emptyset \text{ } \forall n \in \mathbb{Z}$,\newline  (iii) $\mu_0(\mathbbm{1}_G)=0$ and (iv) $\mu_0$ has values only in $\mathbb{Z}$.
For such $\mu_0$ denote $\mu_0^{-1}(n):= G_n(\mu)$ which is of course  countable set. In particular we have $\coprod \limits_{n \in \mathbb{Z}} G_{n}(\mu)= G$.

Now for each $n \in \mathbb{Z}$ biject $G_{n}(\mu)$ with $H \times \{n\} \subset \mathcal{L}$.  As we will see below, we can in-fact assume without loss of generality, following the countability of $G$, that $G_{n}(\mu)$ is infinite. So we now have identified $G$ with $\mathcal{L}$. Denote by $\Psi^{\mu}$ the bijection between $G$ and $\mathcal{L}$.  We will  denoted the left action of the group on itself by $l_g$ for the action of $g$, and the induced action on $\mathcal{L}$ by $\tilde{l}_g$. Following remark \eqref{br} we know that $\mu_0$ and $G$ satisfy condition \eqref{root condition} of lemma \eqref{basiclemma} with some bound $B$. Further,  by definition the left action of $g$ on $\Psi^{\mu}(x)$ for some $x \in G$ is $\Psi^{\mu}(l_g\cdot x)$. 

By construction we have for any $x\text{, }y\text{, }g \in G$: \begin{equation}\label{seq}|\mu_0(l_g \cdot x)-\mu_0(l_g \cdot  y)|= d_2(\Psi^{\mu}(l_g \cdot x), \Psi^{\mu}(l_g \cdot y))= d_2(\tilde{l}_g \cdot \Psi^{\mu}( x),\tilde{l}_g \cdot  \Psi^{\mu}( y)) \end{equation}

Combining \eqref{seq} and \eqref{root condition} (implemented to $\mu_0$ and $G$ for the bound $B$) we get 

\begin{equation}\label{eeq} | d_2(\tilde{l}_g \cdot \Psi^{\mu}( x),\tilde{l}_g \cdot  \Psi^{\mu}( y)) - d_2(\Psi^{\mu}( x),  \Psi^{\mu}( y)) | \leq B \end{equation}

Now since the metric $d$ is bounded by 1 from $d_2$ we get $$  | d(\tilde{l}_g \cdot \Psi^{\mu}( x),\tilde{l}_g \cdot  \Psi^{\mu}( y)) - d(\Psi^{\mu}( x),  \Psi^{\mu}( y)) | \leq B+2$$

which means that $G$ acts on $(\mathcal{L},d)$ via quasi-isometries. Now,  by simply following the construction we see that actually the image of $G$ is in $QI^{h}(\mathcal{L},d)$. To see the second part of the lemma we observe that we can reconstruct $\mu$ from $\Psi^{\mu}$ by choosing any point $x \in \mathcal{L}$ and simply iterate the values of $h$ on the induced action of the group element on the chosen point (this can be seen transparently by following the construction). In particular the value does not depend on a choice of  representative (since the bounded error will vanish in the homogenization). Summing up we have an injection as required.

\end{proof}

\subsection{Family 2: Quasimorphisms that comes from monotonicity of $h$}

We now move to the next example which we give mainly for the sake of completeness of the scope of the examples. Here we use Theorem \eqref{maintheorem} a bit differently, we have the following set up.

Assume that a set $G \subseteq Aut(X)$ acts on $(X, h, A)$ (see definition \eqref{triple}) s.t.  \newline a) $ \forall g \in G , x, y \in X $,  $ h(gx) \geq h(gy) \iff  h(x) \geq h(y) $  \newline b) Elements of $G$ preserve the level sets of $h$.\newline c) The error function $b$ equals zero.   Then, 


\begin{theorem} Under the above assumptions \label{arot}$G$ is a group. The action of $G$ on the triple satisfies all the condition of Theorem \eqref{maintheorem} so lemma  \eqref{basiclemma} applied to this case gives that  $ \mu$  is a  quasimorphism on $G$. Further for the case $A= \mathbb{Z}$,  $\mu$ comes from an action on the circle.
\end{theorem}

\begin{proof}

We first remark, that it is very simple to define  $G \subseteq Aut(X)$ such that all the assumptions are kept but the preservation of the level sets of $h$,  and $G$ will not be group. This being said,  once we add the assumption then trivially we have that $G$ is a group. Still,  in order to use Theorem \eqref{maintheorem} we must show that for all $g\in G$ we have $h(g(F_0)) \subseteq [r, r+C_0]$ for some $r$,  which depends on $g$,  and $C_0$ which is not. The proof for $A=\mathbb{Z}$ and the general case is the same so we will give a proof for $A=\mathbb{Z}$. We will see that $C_0=1$. We first recall that any element of $G$ is monotone with respect to $h$ in the sense above. Let us denote $r+j := \inf \{h(g(F_0) \cap F_j)\}$ for some $0 \leq r < 1$ and some $j \in \mathbb{Z}$. So we need to show that \begin{equation}\label{ub} h(g(F_0)) \subseteq [r+j, r+j+1] \end{equation}
Assume not, let $z \in F_0 \text{ },a \in X$ be such that $g(z)=a$ and $h(a) >r+j+1$. Write\footnote{For simplicity of notation, in this proof,  we will denote by $\alpha_k$ the k-action of $k \in \mathbb{Z}$.} $\alpha_{\mathbbm{1}}(b)=a$ for $b \in F_j$. We claim that $b$ is not covered by $g(F_0)$. If so, let $x \in F_0$ such that $g(x)=b$ Applying $\alpha_{\mathbbm{1}}$ both sides we get $\alpha_{\mathbbm{1}} g(x)= \alpha_{\mathbbm{1}}(b) \Leftrightarrow g(\alpha_{\mathbbm{1}}(x))=a$ which means (recall that $g$ is a bijection) $\alpha_{\mathbbm{1}}(x) \in F_\mathbbm{1}$ is the pre-image of $a$. This is a contradiction. Thus $b$ is not covered by $F_0$. 

So $b$ must be covered by $F_{-\mathbbm{1}}$. On the other hand $h(b)>r+j$. This means that there is an element say $k$
in $g(F_0)\cap F_j$ such that  $r+j \leq h(k) <h(b)$. But $k$ came from $F_0$, while $b$ is covered from $F_{-\mathbbm{1}}$. This contradicts the monotonicity of $g$ with respect to $h$. Thus \eqref{ub} is proved, which is the main condition of Theorem \eqref{maintheorem}.

To show that our  quasimorphism comes from an action on the circle we argue in a standard way. Let us denote by $\mathcal{F}_k := h(F_k)$ and define the subset $\mathcal{F} := \bigcup_{k \in \mathbb{Z}} \mathcal{F}_k \subseteq \mathbb{R}$. Since $G$ preserve the level sets there is an induced action  of $G$ on $  \mathcal{F}$, and further since, by assumption the bounded error function $b$ is zero it follows that there is a translation action of $\mathbb{Z}$ action. Lastly the action of $G$ on $\mathcal{F}$ is monotone. It is a standard fact that such an action extends to the whole real line. We thus get that the image of $G$ is in $Homeo^+_\mathbb{Z}(\mathbb{R})$, the group of monotone orientation preserving homeomorphisms of the real line. Now noting that the constructed homogeneous quasimorphism comes from iterating the values of $h(x_0)$ with respect to the action of the group element we get that  this is exactly the translation number defined on $Homeo^+_\mathbb{Z}(\mathbb{R})$, in particular the  quasimorphism  comes form an action on the circle.

\end{proof}



\section{Diffeomorphism Groups and Discussion} 

We would like to state the following remarks relating to subsequent research and diffeomorphism groups. 


As for diffeomorphism groups we propose the following set-up.

Assume $X=  \coprod  \limits_{n \in \mathbb {Z} }F_{n} \subseteq \mathbb{R}^{k+1}$ is a smooth $k$ manifold, path connected, unbounded.  In particular,  carries an induced Riemannian metric, so we have a notion of length of paths, and volume on $X$. Assume also that the closure
 of $F_{n}$ intersect only the closure of  $F_{n-1}$ and $F_{n+1}$ and the action of $\mathbb{Z}$ on $X$ is via isometries (recall that the action of $n$,  $ \alpha_{n}$,   satisfies that $ \alpha_{n} : F_{0} \xrightarrow {\cong} F_{n}$). Further we choose $U$ to be an open submanifold of $F_0$ which exhaust the (finite) volume of $F_0$ up to a given constant and has the same (finite) diameter as $F_0$. Finally,  choose $U$ such that it has a smooth boundary.

 Now let  $G' \subseteq Diff_\mathbb{Z}(X)$,  $h_{0}: X \rightarrow \mathbb{R}^{\geq 0}$ be a smooth, say, function s.t.  for all $g \in G'$  and fixed $0< \varepsilon$

 \begin{equation} | \int \limits_{g(F_{0})}h_{0}dm | \leq B\end{equation} 
 
 \begin{equation} \varepsilon < \inf \limits_{x \in Int(g(U)} \sup \{vol(B_{r}(x)) | \widebar {B_{r}(x)} \subseteq \widebar {g(U)}, B_{r}(x) \text { } is \text { } an\text{ } open \text { }ball  \text { } \underline{contains} \text { }x \}  \end{equation}

For example: $h_{0} $ can be taken to be periodic with respect to the  $ F_{n}'s $
 so the elements of $G'$ respect the period up to some error where the simplest example is when $h_{0}$ is just a constant function. 
 
 We define $h: X \rightarrow \mathbb{R}$ as follows: we choose a reference point $x_{0} \in U$ and define 
\begin{equation} h(x)=\pm \inf \limits_{\gamma} \{ | \int \limits_{\gamma} h_{0} dx | ;   \gamma  \text{ is a path which connects } x_{0} \text{ } to \text{ } x \} \end{equation}
 where $dx$ is  length element of the metric and the sign is determined to be compatible with the labeling of the fundamental domains. For "simple enough" fundamental domain $F_0$ such an $h$ will make $(X, h, \mathbb{Z})$ into a triple in the sense of \eqref{triple}. Finally, note  that the set of groups contained in $G'$ which satisfy the above conditions is not empty, since the translations group is such a group. Even further in some cases, as we saw in Example 1 of the previous section, we have a family of such groups. Thus by Zorn's lemma we have at least one maximal group $G$ with the conditions above. For such $G$ we have:

  \begin{theorem} 
 
  For all $g \in G $ we have $h(g(U)) \subseteq [r, r+ C_{0}] $ where $C_{0}$ depends only on 
 $h$ and $\varepsilon$ and we iterate only the points of $U$. Thus the formula $ \mu (g)= h(g \cdot x_{0})-h(x_{0})$ defines a  quasimorphism on 
$G$ and $ \mu ^{h}$ the homogenization of $\mu$  a non zero homogeneous quasimorphism on $G$.

\end{theorem}

The disadvantage here,  is that we do not know what are the properties  of $G$ and how much  interesting it is,  being determined in a crude way. 

\textbf{2.}  The basic picture that stems from lemma \eqref{qil},  according to \cite{Ha},  seems to relate to \cite{Ma} where the results there seem to relate to at least one aspect of implications of the lemma. In \cite{Ma},  quasimorphisms are indeed studied on finitely presented groups via, more or less, quasi-isometric action on trees. It seems that by using lemma \eqref{qil} as a starting point,  and results of \cite{BS-H4}, 
new interpretation to \cite{Ma} can be given. An important point is that this work can also be used to study similar ideas where we consider countable groups into diffeomorphism groups using quasi isometries as carriers of data on quasimorphisms on the groups. I should mention that partial motivation could come from \cite{Polt} and mentioned works therein.


\

 \end{document}